\documentclass[11pt,reqno]{amsart}

\usepackage[T1]{fontenc}
\usepackage[utf8]{inputenc}
\usepackage{lmodern}
\usepackage{amsmath,amssymb,amsthm,mathtools}
\usepackage{booktabs,array}
\usepackage{enumitem}
\usepackage{microtype}
\usepackage[hidelinks]{hyperref}
\usepackage[nameinlink,noabbrev]{cleveref}

\allowdisplaybreaks
\setlist[enumerate]{label=(\arabic*),leftmargin=2.2em,itemsep=2pt,topsep=4pt}

\newtheorem{theorem}{Theorem}[section]
\newtheorem{lemma}[theorem]{Lemma}
\newtheorem{proposition}[theorem]{Proposition}
\newtheorem{corollary}[theorem]{Corollary}
\theoremstyle{definition}

\newcommand{\cthreshold}{c_{\mathrm{sc}}}
\newcommand{\bK}[1]{\overleftrightarrow{K}_{#1}}
\newcommand{\dP}{\overrightarrow{P}_{3}}
\newcommand{\dV}{\overrightarrow{V}_{3}}
\newcommand{\dL}{\overrightarrow{\Lambda}_{3}}
\newcommand{\dunion}{\mathbin{\dot\cup}}
\newcommand{\Sym}{\operatorname{Sym}}
\newcommand{\Aut}{\operatorname{Aut}}
\newcommand{\can}{\operatorname{can}}

\title[Self-complementary completions on six vertices]
{Self-complementary completions of six-vertex digraphs}

\author{Xinan Dai}
\address{Key Laboratory for Information Science of Electromagnetic Waves, College of Future Information Technology, Fudan University, Shanghai, China}
\curraddr{Department of Artificial Intelligence, School of Engineering, Westlake University, Hangzhou, China}
\email{xndai23@m.fudan.edu.cn}

\author{Wenhao Deng}
\address{Department of Artificial Intelligence, School of Engineering, Westlake University, Hangzhou, China}
\email{dengwenhao@westlake.edu.cn}

\author{Yingdong Shi}
\address{School of Information Science and Technology, ShanghaiTech University, Shanghai, China}
\email{shiyd2023@shanghaitech.edu.cn}

\author{Tailin Wu}
\address{Department of Artificial Intelligence, School of Engineering, Westlake University, Hangzhou, China}
\email{wutailin@westlake.edu.cn}

\author{Yuchen Yang}
\address{Department of Artificial Intelligence, School of Engineering, Westlake University, Hangzhou, China}
\email{yangyuchen@westlake.edu.cn}

\date{27 July 2026}

\begin{document}

\begin{abstract}
Let \(\cthreshold(n)\) be the largest integer \(q\) such that every loopless digraph on \(n\) vertices with at most \(q\) arcs is isomorphic to a spanning subdigraph of a self-complementary digraph of order \(n\). We prove that \(\cthreshold(6)=7\). The upper bound is witnessed by
\[
\bK{3}\dunion (x\longrightarrow y\longrightarrow z),
\]
and follows from a direct argument with a self-complementing permutation. We also determine the complete eight-arc obstruction layer: it consists of five isomorphism classes, or three after converse digraphs are identified. All five are arc-minimal. Each nevertheless packs with an isomorphic copy of itself, so ordinary packing is strictly weaker than same-order self-complementary completion already at this first failure layer.
\end{abstract}

\maketitle

\section{Introduction}\label{sec:introduction}

All digraphs in this paper are finite and loopless. Opposite arcs are allowed, but repeated arcs are not. The complement \(\overline D\) of a digraph \(D\) has the same vertex set as \(D\), and an ordered pair of distinct vertices is an arc of \(\overline D\) exactly when it is not an arc of \(D\). A digraph is \emph{self-complementary} if it is isomorphic to its complement.

An \(n\)-vertex digraph \(D\) will be called \emph{completable} if it is isomorphic to a spanning subdigraph of some self-complementary digraph of order \(n\). We study the threshold
\begin{equation}\label{eq:csc-def}
\cthreshold(n)=\max\left\{q:
\begin{array}{c}
\text{every \(n\)-vertex digraph with at most \(q\) arcs}\\[-2pt]
\text{is completable}
\end{array}
\right\}.
\end{equation}
The requirement that the completion have the same order is essential: no new vertices may be added.

Benhocine and Wojda proved that \(\cthreshold(n)\ge n\) and conjectured a stronger bound, subject to a specified exceptional family \cite{BW85}. Zhang and Wei later found counterexamples on four vertices. They also proved that every \(p\)-vertex digraph with at most \(p+1\) arcs is completable, apart from a finite family of four-vertex, five-arc digraphs \cite{ZW00}. Two later formulations are relevant here: Woźniak asked whether every \(n\)-vertex digraph with at most \(2n-4\) arcs is completable \cite[p.~75, Conjecture~5.25]{Woz97}, while Zhang and Wei proposed a repaired \(2n-3\) statement allowing the previously known exceptions \cite[Conjecture~2]{ZW00}.

Both formulations fail at order six. A smallest-looking witness already has a transparent structure: a bidirected triangle, disjoint from a directed path of length two. The triangle forces its image and preimage under any candidate self-complementing permutation to be independent triples; the two endpoints of the path are then forced into both triples, leaving the middle vertex fixed. This gives the exact threshold.

\begin{theorem}\label{thm:main-threshold}
For six vertices,
\[
\cthreshold(6)=7.
\]
In particular,
\begin{equation}\label{eq:D0}
D_0=\bK{3}\dunion\dP
\end{equation}
has eight arcs and is not completable.
\end{theorem}

Here \(\dP\) is a consistently directed path on three vertices. We write \(\dV\) and \(\dL\) for the two orientations of a two-edge path with, respectively, a central source and a central sink. Let \(T_4^+\) be obtained from a directed triangle by adding a vertex that sends an arc to all three triangle vertices, and let \(T_4^-\) be its converse.

The example in \eqref{eq:D0} is one of only five obstructions at the first failing size.

\begin{theorem}\label{thm:main-classification}
Up to isomorphism, the noncompletable six-vertex digraphs with eight arcs are
\[
\bK{3}\dunion\dP,
\qquad
\bK{3}\dunion\dV,
\qquad
\bK{3}\dunion\dL,
\qquad
\bK{2}\dunion T_4^+,
\qquad
\bK{2}\dunion T_4^-.
\]
There are \(480\) labelled examples. Every member of the list is arc-minimal, and the list reduces to three classes when converse digraphs are identified.
\end{theorem}

There is a related, weaker condition. Two digraphs \(D_1,D_2\) of order \(n\) \emph{pack} if a bijection \(\pi:V(D_1)\to V(D_2)\) makes \(A(D_2)\) and \(\pi(A(D_1))\) disjoint. If \(D\) is contained in a self-complementary digraph \(S\) and \(\sigma(S)=\overline S\), then \(D\) and \(\sigma(D)\) pack. The converse need not hold.

\begin{theorem}\label{thm:main-packing}
Each of the five digraphs in \cref{thm:main-classification} packs with an isomorphic copy of itself.
\end{theorem}

\section{A criterion for completion}\label{sec:criterion}

Let \(V\) be a finite set and put
\[
\Omega(V)=\{(u,v)\in V\times V:u\ne v\}.
\]
A permutation \(\sigma\in\Sym(V)\) acts on ordered pairs by
\[
\sigma(u,v)=(\sigma(u),\sigma(v)).
\]
If \(\sigma(S)=\overline S\), then repeated application of \(\sigma\) alternates arcs and nonarcs along every orbit in \(\Omega(V)\). Each such orbit therefore has two parity classes, and a prescribed subdigraph may occupy at most one of them. This observation gives the following criterion, in the form used below.

\begin{lemma}[Alternating-orbit criterion]\label{lem:criterion}
A digraph \(D\) is completable if and only if there is a permutation \(\sigma\in\Sym(V(D))\) such that
\begin{enumerate}
\item every nontrivial cycle of \(\sigma\) has even length, and \(\sigma\) has at most one fixed point;
\item for every \(j\ge0\),
\begin{equation}\label{eq:odd-disjoint}
A(D)\cap \sigma^{2j+1}(A(D))=\varnothing.
\end{equation}
\end{enumerate}
Equivalently, on each \(\langle\sigma\rangle\)-orbit in \(\Omega(V(D))\), the arcs of \(D\) lie in at most one of the two alternating parity classes.
\end{lemma}

\begin{proof}
Suppose first that \(D\subseteq S\) after relabelling, where \(S\) is self-complementary, and choose \(\sigma\) with \(\sigma(S)=\overline S\). Membership in \(A(S)\) changes each time \(\sigma\) is applied. Hence every orbit of \(\sigma\) on \(\Omega(V(D))\) has even length: an odd orbit would return to its initial ordered pair after changing its membership status an odd number of times. It follows at once that every odd power of \(\sigma\) sends arcs of \(S\), and therefore arcs of \(D\), to nonarcs of \(S\), which proves \eqref{eq:odd-disjoint}.

It remains to read the even-orbit condition on the vertex cycles of \(\sigma\). A vertex cycle of odd length \(r>1\) would make the orbit of \((v,\sigma(v))\) have the same odd length \(r\). Two fixed points \(u,v\) would make \((u,v)\) an orbit of length one. Thus every nontrivial vertex cycle is even and there is at most one fixed point, proving~(1).

Conversely, assume that \(\sigma\) satisfies~(1) and~(2). Condition~(1) implies that every orbit on ordered pairs has even length. Indeed, two vertices in the same nontrivial cycle give an orbit whose length divides that even cycle length and, for distinct vertices in one cycle, is in fact the cycle length; vertices in different cycles give an orbit of length equal to the least common multiple of the two cycle lengths. Since there cannot be two fixed points, at least one of those lengths is even.

Split each ordered-pair orbit into its even and odd positions. By~(2), the arc set of \(D\) meets at most one parity class on each orbit. On an occupied orbit choose the class containing the prescribed arcs; on an empty orbit choose either class. Let \(S\) be the digraph whose arc set is the union of the chosen classes. Then \(D\subseteq S\), while \(\sigma\) exchanges the chosen and unchosen classes on every orbit. Hence \(\sigma(S)=\overline S\), so \(S\) is a self-complementary completion of \(D\).
\end{proof}

On six vertices, condition~(1) permits only the cycle types
\[
(6),\qquad (4,2),\qquad (2,2,2).
\]
Their numbers are
\[
(6-1)!=120,
\qquad
\frac{6!}{4\cdot2}=90,
\qquad
\frac{6!}{2^3 3!}=15,
\]
respectively. Thus there are exactly \(225\) admissible permutations at order six.

\section{A small obstruction}\label{sec:elementary}

Let \(A=\{a_0,a_1,a_2\}\) span a bidirected triangle, and let \(x,y,z\) form a directed path \(x\to y\to z\). There are no arcs between the two components. The resulting digraph is \(D_0\) from \eqref{eq:D0}. Here an \emph{independent} set means a set containing no arc in either direction between two of its vertices.

\begin{proposition}\label{prop:D0}
The digraph \(D_0\) is not completable.
\end{proposition}

\begin{proof}
Assume that a permutation \(\sigma\) satisfies \cref{lem:criterion}. The six ordered pairs between distinct vertices of \(A\) are all arcs of \(D_0\). Therefore \(\sigma(A)\) must be independent: otherwise an arc inside \(\sigma(A)\) would lie in both \(A(D_0)\) and \(\sigma(A(D_0))\), contrary to \eqref{eq:odd-disjoint} with \(j=0\).

The order \(L\) of \(\sigma\) is even, since \(\sigma\) has only even cycles on six vertices. Thus \(\sigma^{-1}=\sigma^{L-1}\), where \(L-1\) is odd. Applying \eqref{eq:odd-disjoint} to this odd power shows in the same way that \(\sigma^{-1}(A)\) is independent.

The independent triples of \(D_0\) are exactly
\begin{equation}\label{eq:independent-triples}
\{a_i,x,z\},\qquad i=0,1,2.
\end{equation}
Indeed, such a triple contains at most one vertex of the bidirected triangle, and \(x,z\) are the only nonadjacent pair among the path vertices. Hence, for some \(i,j\in\{0,1,2\}\),
\[
\sigma(A)=\{a_i,x,z\},
\qquad
\sigma^{-1}(A)=\{a_j,x,z\}.
\]
Set \(B=A\cup\{x,z\}\). The first equality gives \(\sigma(A)\subseteq B\). Applying \(\sigma\) to the second gives
\[
A=\{\sigma(a_j),\sigma(x),\sigma(z)\},
\]
so \(\sigma(\{x,z\})\subseteq A\). Consequently \(\sigma(B)\subseteq B\). Since \(\sigma\) is a bijection and \(|B|=5\), we have \(\sigma(B)=B\). The remaining vertex \(y\) is therefore fixed, contradicting condition~(1) of \cref{lem:criterion} at even order.
\end{proof}

The argument depends on the unique nonadjacent pair \(\{x,z\}\) in the path component. Both the image and the preimage of the bidirected triangle are forced to contain this pair, and the middle vertex is left outside an invariant five-set.

\section{The six-vertex threshold}\label{sec:threshold}

\begin{proof}[Proof of \cref{thm:main-threshold}]
Zhang and Wei's theorem implies that every six-vertex digraph with at most seven arcs is completable, since their exceptions have order four. Hence \(\cthreshold(6)\ge7\). On the other hand, \(D_0\) has eight arcs and is not completable by \cref{prop:D0}, so \(\cthreshold(6)\le7\). Therefore \(\cthreshold(6)=7\).
\end{proof}

\begin{corollary}\label{cor:conjectures}
The digraph \(D_0\) disproves both Woźniak's \(2n-4\) conjecture and Zhang--Wei's repaired \(2n-3\) conjecture. Moreover, six is the first order at which the exception-free \(2n-4\) statement fails.
\end{corollary}

\begin{proof}
At \(n=6\), the digraph \(D_0\) has \(8=2n-4\) arcs, is not completable, and is not one of Zhang and Wei's four-vertex exceptions. It therefore contradicts both statements.

For \(n\le4\), one has \(2n-4\le n\), so the result of Benhocine and Wojda covers the proposed range. At \(n=5\), the bound is \(2n-4=6=n+1\), and Zhang and Wei's theorem applies because its exceptional family has order four. Thus no smaller order can fail.
\end{proof}

\section{The eight-arc obstructions}\label{sec:classification}

The direct argument above identifies one obstruction. At order six the alternating-orbit criterion is also small enough for an exhaustive check: only the \(225\) admissible permutations need to be considered. The resulting five classes are listed in \cref{tab:classes}; precise details of the enumeration are given in \cref{app:enumeration}.

\begin{table}[ht]
\caption{The five isomorphism classes of eight-arc obstructions. Counts refer to a fixed labelled six-element vertex set.}
\label{tab:classes}
\centering
\small
\setlength{\tabcolsep}{4pt}
\begin{tabular}{@{}lrrrl@{}}
\toprule
Digraph & \(|\Aut(D)|\) & Labelled copies & Packing maps & Converse \\
\midrule
\(\bK{3}\dunion\dP\) & 6  & 120 & 36 & itself \\
\(\bK{3}\dunion\dV\) & 12 & 60  & 36 & \(\bK{3}\dunion\dL\) \\
\(\bK{3}\dunion\dL\) & 12 & 60  & 36 & \(\bK{3}\dunion\dV\) \\
\(\bK{2}\dunion T_4^+\) & 6 & 120 & 48 & \(\bK{2}\dunion T_4^-\) \\
\(\bK{2}\dunion T_4^-\) & 6 & 120 & 48 & \(\bK{2}\dunion T_4^+\) \\
\bottomrule
\end{tabular}
\end{table}

\begin{proof}[Proof of \cref{thm:main-classification}]
Fix \(V=\{0,1,2,3,4,5\}\) and order the \(30\) possible arcs lexicographically. For each admissible permutation \(\sigma\), decompose \(\Omega(V)\) into \(\langle\sigma\rangle\)-orbits and split each orbit into its two alternating classes. By \cref{lem:criterion}, an arc set \(M\subseteq\Omega(V)\) is completable if and only if, for at least one admissible \(\sigma\), it meets at most one parity class on every orbit.

This predicate was evaluated for all
\[
\binom{30}{8}=5{,}852{,}925
\]
eight-arc sets. Exactly \(480\) fail for every admissible permutation. For each rejected set \(M\), define
\[
\can(M)=\min_{\pi\in S_6}\pi(M)
\]
with respect to the fixed lexicographic bit order. Two labelled digraphs have the same canonical mask exactly when they are isomorphic. The \(480\) rejected masks yield five canonical masks, represented by the five digraphs in the theorem.

Their orbit sizes under \(S_6\) are
\[
120,\quad 60,\quad 60,\quad 120,\quad 120,
\]
which sum to \(480\) and agree independently with \(6!/|\Aut(D)|\) using the automorphism orders in \cref{tab:classes}. Taking converses fixes the \(\bK{3}\dunion\dP\) class and exchanges the two remaining pairs.

For each representative \(D\) and each arc \(e\in A(D)\), the computation also records an admissible permutation satisfying \cref{lem:criterion} for \(D-e\). Thus every one-arc deletion is completable. Since every spanning subdigraph of a completable digraph is again completable, each of the five obstructions is arc-minimal.
\end{proof}

The complete source code, its output, and the \(40\) deletion witnesses used in the last paragraph are included as ancillary files. As a calibration independent of the six-vertex totals, the same program at order four returns \(288\) labelled five-arc obstructions in \(14\) isomorphism classes, matching the exceptional family described by Zhang and Wei.

\section{Ordinary packing}\label{sec:packing-separation}

A packing map for a digraph \(D\) is a permutation \(\pi\in S_6\) such that
\[
A(D)\cap\pi(A(D))=\varnothing.
\]
The permutations in \cref{tab:packing-certificates} give packing maps for the labelled representatives in \cref{app:representatives}.

\begin{table}[ht]
\caption{Packing maps, written as \((\pi(0),\ldots,\pi(5))\).}
\label{tab:packing-certificates}
\centering
\small
\begin{tabular}{@{}ll@{}}
\toprule
Digraph & Packing map \\
\midrule
\(\bK{3}\dunion\dP\) & \((0,3,5,1,4,2)\) \\
\(\bK{3}\dunion\dV\) & \((0,3,5,1,4,2)\) \\
\(\bK{3}\dunion\dL\) & \((0,3,5,1,4,2)\) \\
\(\bK{2}\dunion T_4^+\) & \((0,2,1,3,5,4)\) \\
\(\bK{2}\dunion T_4^-\) & \((0,2,1,5,4,3)\) \\
\bottomrule
\end{tabular}
\end{table}

\begin{proof}[Proof of \cref{thm:main-packing}]
For each row of \cref{tab:packing-certificates}, direct substitution in the corresponding arc list gives
\[
A(D)\cap\pi(A(D))=\varnothing.
\]
Thus every obstruction packs with an isomorphic copy of itself.
\end{proof}

The difference from self-complementary completion lies in the coherence imposed by one permutation. Ordinary packing asks only for one disjoint relabelling of \(D\). The criterion in \cref{lem:criterion}, by contrast, requires every odd power of a single admissible permutation to avoid the prescribed arcs. The five obstructions satisfy the former condition but not the latter. They do not improve the general ordinary-packing bounds studied by Görlich and Żak \cite{GZ10}.

\appendix

\section{Labelled representatives}\label{app:representatives}

All representatives use \(V=\{0,1,2,3,4,5\}\). Put
\[
K_3^{\leftrightarrow}=\{(i,j):0\le i,j\le2,\ i\ne j\}.
\]
Then
\begin{align*}
A(\bK{3}\dunion\dP)
  &=K_3^{\leftrightarrow}\cup\{(3,4),(4,5)\},\\
A(\bK{3}\dunion\dV)
  &=K_3^{\leftrightarrow}\cup\{(4,3),(4,5)\},\\
A(\bK{3}\dunion\dL)
  &=K_3^{\leftrightarrow}\cup\{(3,4),(5,4)\},\\
A(\bK{2}\dunion T_4^+)
  &=\{(0,1),(1,0),(2,3),(3,4),(4,2),\\[-2pt]
  &\hspace{8em}(5,2),(5,3),(5,4)\},\\
A(\bK{2}\dunion T_4^-)
  &=\{(0,1),(1,0),(2,3),(3,4),(4,2),\\[-2pt]
  &\hspace{8em}(2,5),(3,5),(4,5)\}.
\end{align*}

With the lexicographic arc order \((0,1),(0,2),\ldots,(5,4)\), the corresponding canonical decimal masks are shown in \cref{tab:canonical-masks}. They are included only to make the isomorphism reduction easy to reproduce.

\begin{table}[ht]
\caption{Canonical masks in the fixed \(30\)-bit encoding.}
\label{tab:canonical-masks}
\centering
\small
\begin{tabular}{@{}lr@{}}
\toprule
Class & Canonical mask \\
\midrule
\(\bK{3}\dunion\dP\) & 1,251,536 \\
\(\bK{3}\dunion\dV\) & 123,045 \\
\(\bK{3}\dunion\dL\) & 3,686,697 \\
\(\bK{2}\dunion T_4^+\) & 35,928,208 \\
\(\bK{2}\dunion T_4^-\) & 1,659,464 \\
\bottomrule
\end{tabular}
\end{table}

\section{Enumeration}\label{app:enumeration}

Represent an arc set by a \(30\)-bit integer in the arc order above. For an admissible permutation \(\sigma\), let
\[
\mathcal C_\sigma=\{(O_0,O_1):O\text{ is an orbit of }\langle\sigma\rangle\text{ on }\Omega(V)\},
\]
where \(O_0\) and \(O_1\) are the two alternating classes of \(O\). A mask \(M\) is accepted for \(\sigma\) precisely when
\begin{equation}\label{eq:mask-predicate}
\text{for every }(O_0,O_1)\in\mathcal C_\sigma,
\qquad
M\cap O_0=\varnothing\quad\text{or}\quad M\cap O_1=\varnothing.
\end{equation}
It is an obstruction precisely when \eqref{eq:mask-predicate} fails for every admissible \(\sigma\).

The following pseudocode specifies the search and the isomorphism quotient.

\begingroup
\footnotesize
\begin{verbatim}
constraints = []
for sigma in permutations({0,...,5}):
    if cycle_type(sigma) in {(6),(4,2),(2,2,2)}:
        constraints.append(orbit_parity_pairs(sigma))

obstructions = []
for M in weight_eight_masks(30):
    if all(any((M intersects E) and (M intersects O)
               for (E,O) in C)
           for C in constraints):
        obstructions.append(M)

canonical_counts = frequency_table(
    min(relabel(M, pi) for pi in SymmetricGroup(6))
    for M in obstructions
)
\end{verbatim}
\endgroup

For an accepted pair \((M,\sigma)\), a completion can be recovered orbit by orbit: choose the parity class met by \(M\) on every occupied orbit, and either class on an empty orbit. The implementation uses this construction to record an accepting permutation for every one-arc deletion of the five representatives. All operations are on finite integer masks; no random or floating-point step enters the predicate.

The independently rerun computation produced the following totals:
\[
\begin{aligned}
&225\text{ admissible permutations},
&&5{,}852{,}925\text{ candidates},\\
&480\text{ labelled obstructions},
&&5\text{ isomorphism classes}.
\end{aligned}
\]
It also returned \(36,36,36,48,48\) packing maps for the five representatives in the order displayed in \cref{tab:classes}, and \(288\) labelled obstructions in \(14\) classes in the four-vertex calibration.

\section*{Statement on AI-assisted discovery}
The TARS agent system was used in the initial search for small counterexamples and in organizing the first finite enumeration. Xinan Dai completed proof reconstruction, literature review and manuscript writing and verified all mathematical results manually.

\bibliographystyle{amsplain}
\bibliography{references}

@article{BW85,
  author  = {Benhocine, A. and Wojda, A. P.},
  title   = {On self-complementation},
  journal = {Journal of Graph Theory},
  volume  = {9},
  number  = {3},
  pages   = {335--341},
  year    = {1985},
  doi     = {10.1002/jgt.3190090305},
  url     = {https://doi.org/10.1002/jgt.3190090305},
  note    = {doi: \href{https://doi.org/10.1002/jgt.3190090305}{10.1002/jgt.3190090305}}
}

@book{Woz97,
  author    = {Wo{\'z}niak, Mariusz},
  title     = {Packing of Graphs},
  series    = {Dissertationes Mathematicae},
  volume    = {362},
  publisher = {Institute of Mathematics, Polish Academy of Sciences},
  address   = {Warsaw},
  pages     = {1--78},
  year      = {1997},
  url       = {https://pldml.icm.edu.pl/pldml/element/bwmeta1.element.zamlynska-b3a7d237-fb32-41fe-a7ec-fd4f3663ca9e},
  note      = {\href{https://pldml.icm.edu.pl/pldml/element/bwmeta1.element.zamlynska-b3a7d237-fb32-41fe-a7ec-fd4f3663ca9e}{PLDML record}}
}

@article{ZW00,
  author  = {Zhang, Yunqing and Wei, Xiansun},
  title   = {{Yi lei youxiangtu de keruqianxing} [{Embeddability} of a class of digraphs]},
  journal = {Journal of Shaanxi Normal University (Natural Science Edition)},
  volume  = {28},
  number  = {4},
  pages   = {19--22},
  year    = {2000},
  issn    = {1672-4291},
  note    = {In Chinese; \href{https://www.cqvip.com/doc/journal/4873736}{CQVIP record}},
  url     = {https://www.cqvip.com/doc/journal/4873736}
}

@article{GZ10,
  author  = {G{\"o}rlich, Agnieszka and {\.Z}ak, Andrzej},
  title   = {On Packable Digraphs},
  journal = {SIAM Journal on Discrete Mathematics},
  volume  = {24},
  number  = {2},
  pages   = {552--557},
  year    = {2010},
  doi     = {10.1137/090748056},
  url     = {https://doi.org/10.1137/090748056},
  note    = {doi: \href{https://doi.org/10.1137/090748056}{10.1137/090748056}}
}

\end{document}